\newif\ifArXiV

\ArXiVtrue	

\ifArXiV
\documentclass{article}
\usepackage{fullpage}

\else
\RequirePackage{fix-cm}
\documentclass[smallextended]{svjour3} 
\smartqed 
\fi

\usepackage{xcolor}
\usepackage{graphicx}
\usepackage{booktabs}
\usepackage[ruled,vlined,linesnumbered]{algorithm2e}
\usepackage{hyperref}
\usepackage{tikz}
\usepackage{hf-tikz}
\usepackage{tikzsymbols}
\usepackage{csquotes}
\usepackage{lscape}

\usepackage{amsmath, amssymb,amsfonts}

\usetikzlibrary{calc,arrows,positioning,backgrounds,decorations.pathreplacing}
\usepackage{mathtools}

\tikzset{%
	false-child/.style = {-stealth,dashed},
	true-child/.style = {-stealth,solid},
	c/.style = {draw,solid,minimum width=2em,
		minimum height=1.em},
	r/.style = {rectangle,draw,solid,minimum width=2em,fill=white,
		minimum height=1em},
	l/.style = {rectangle,minimum width=0.5em,fill=white,
		minimum height=1.5em},     
	p/.style = {rectangle,minimum width=0.5em,fill=white,
		minimum height=0.5em},
}

\newcommand{\IA}{IA\xspace}
\newcommand{\OA}{OA\xspace}

\newcommand{\MOLP}{MOLP\xspace}

\newcommand{\MOILP}{MOILP\xspace}
\newcommand{\MOMILP}{MOMILP\xspace}

\newcommand{\eghull}{\mathcal Q^+}

\usepackage{caption}
\usepackage{subcaption}

\usepackage[textwidth=5ex]{todonotes}

\ifArXiV
\usepackage[affil-it]{authblk}
\usepackage[authoryear]{natbib}

\newenvironment{keywords}{\small \textbf{Keywords: }}{}

\usepackage{amsthm}
\newtheorem{theorem}{Theorem}
\newtheorem{definition}[theorem]{Definition}

\newtheorem{lemma}[theorem]{Lemma}
\newtheorem{corollary}[theorem]{Corollary}

\newtheorem{remark}[theorem]{Remark}
\else
\newcommand{\qedhere}{\hfill$\square$}
\fi

\begin{document}

\title{An outer approximation algorithm for multi-objective mixed-integer linear and non-linear programming}
 			\ifArXiV
			\author[1]{Fritz Bökler\thanks{fritz.boekler@uos.de}}
			\author[2]{Sophie N. Parragh\thanks{sophie.parragh@jku.at}}
			\author[2]{Markus Sinnl\thanks{markus.sinnl@jku.at}}
			\author[3]{Fabien Tricoire\thanks{fabien.tricoire@wu.ac.at}}

			\affil[1]{Institute of Computer Science, Osnabrück University, Osnabrück, Germany}
			\affil[2]{Institute of Production and Logistics Management/JKU Business School, 
			Johannes Kepler University Linz, Linz, Austria}
			\affil[3]{Institute for Transport and Logistics Management,Vienna University of Economics and Business, Vienna, Austria}			
			\date{}
			\maketitle
			
			\fi
\maketitle


\begin{abstract}
	
In this paper, we present the 
first outer approximation algorithm for multi-objective mixed-integer linear programming problems with any number of objectives. The algorithm also works for certain classes of non-linear programming problems. It produces the non-dominated extreme points as well as the facets of the convex hull of these points. The algorithm relies on an oracle which solves single-objective weighted-sum problems and we show that 
the required number of oracle calls is polynomial in the number of facets of the convex hull of the non-dominated extreme points in the case of multiobjective mixed-integer programming (MOMILP). Thus, for MOMILP problems for which the weighted-sum problem is solvable in polynomial time, the facets can be computed with incremental-polynomial delay. From a practical perspective, the algorithm starts from a valid lower bound set for the non-dominated extreme points and iteratively improves it. Therefore it can be used in multi-objective branch-and-bound algorithms and still provide a valid bound set at any stage, even if interrupted before converging. Moreover, the oracle produces Pareto optimal solutions, which makes the algorithm also attractive from the primal side in a multi-objective branch-and-bound context. Finally, the oracle can also be called with any relaxation of the primal problem, and the obtained points and facets still provide a valid lower bound set. A computational study on a set of benchmark instances from the literature and new non-linear multi-objective instances is provided.%

\keywords{multi-objective optimization, outer approximation, mixed-integer programming, non-linear programming}

\end{abstract}

\newcommand{\PolySCIP}{\texttt{PolySCIP}\xspace}
\newcommand{\bensolve}{\texttt{bensolve}\xspace}


%
\section{Introduction and motivation}\label{intro} 

Many practical problems involve several, often conflicting objectives, such as profitability or cost versus environmental concerns~\citep{demir2014bi,ramos2014planning,eskandarpour2021multi} or customer satisfaction \citep{braekers2015bi}. This means that there is, in the general case, no single optimal solution which optimizes all objectives simultaneously but a set of trade-off solutions which are better than all other solutions but incomparable among each other. A solution belonging to this set is called \emph{Pareto optimal} and has the property that no objective function value can be improved without deteriorating another. The images of these solutions in the space of objective function values (the \emph{objective space} or \emph{criterion space}) are called \emph{non-dominated} points and they together form the \emph{non-dominated frontier} or \emph{Pareto frontier}. In multi-objective optimization, we usually aim at identifying at least one Pareto optimal solution for each non-dominated point.


Our research aims to further advance general purpose exact methods for bi- and multi-objective (mixed) integer linear programming relying on the branch-and-bound (B\&B) idea. All recent successful implementations rely on \emph{lower bound sets} \citep{ehrgott2007bound}, see, e.g., \citet{gadegaard2019bi,parragh2019branch,forget2022warm,adelgren2021branch} but they are currently either restricted to two objectives or they only address the pure integer and not the mixed integer case. Lower bound sets are sets in the objective space. Depending on the definition, they either contain the non-dominated frontier or a set of points which, together, dominate the non-dominated frontier. They are a natural multi-objective extension of the lower bounds obtained by, e.g., relaxations in single-objective optimization.
In a bi-objective context, they are usually computed "from the inside" with so-called \emph{inner approximation} (\IA) schemes. Only \citet{Forget2020,forget2022warm} use an \emph{outer approximation} (\OA) scheme to compute bound sets "from the outside". The bound sets in these two works are obtained from the \emph{linear programming} (LP) relaxation of the \MOMILP using 
an implementation of Benson's \OA algorithm for \emph{multi-objective linear programming} (\MOLP). 

The \OA algorithm we are proposing can directly obtain such a lower bound set for a \MOMILP and also some classes of multi-objective (mixed) integer non-linear programming problems without the need of a convex relaxation of the considered problem. The algorithm produces the \emph{non-dominated extreme points} and the \emph{facets of the convex hull} of these points for any number of objectives.

The study of efficient algorithms for the computation of extreme points and facets of multi-objective optimization problems is also of independent theoretical interest.
It is well known that even for structurally simple MOLP problems, the number of extreme points can be exponential in the size of the input \citep{R88}.
The same observation is also true for many well-known multi-objective combinatorial optimization (MOCO) problems, e.g., the multi-objective versions of the assignment and shortest path problems. Consequently, a polynomial time algorithm for the problem classes our algorithm can tackle cannot exist.
In recent years, a new complexity measure emerged that is capable of separating some of the complexities of these problem classes.
\cite{bokler18} shows that in an \emph{output-sensitive complexity measure}, MOLP is an easy problem, while the more general \MOMILP remains hard.
The idea is to regard the running time not only as a function of the input, but also of the output size.
This is a very useful measure for multi-objective optimization if the whole set of extreme points, facets, or the whole non-dominated set is sought.
A central theorem in this field of research is that the extreme points of a MOCO problem can be computed efficiently, if the weighted-sum scalarization can be computed efficiently, i.e., in polynomial time \citep{bokler2015output}.
The proof of this theorem roots in the study of MOLP, in which the efficient computation of non-dominated extreme points is possible.
While there are efficient algorithms for the computation of non-dominated facets of an MOLP \citep{bokler18}, no such theorem could yet be proven for the facet computation of more general problem classes. Our work closes this research gap.

\subsection{Problem definition \label{sec:problemdef}}

We are interested in general multi-objective optimization problems (MOP) of the following form and require further assumptions below:
\begin{align*}
\min&\ f(x)\\
\text{s.t. }& x \in \mathcal X,
\end{align*}
where $f\colon \mathcal X \to \mathbb R ^ p$, $p\in \mathbb N$, and $\mathcal X$ is a set.
For this problem we define $\mathcal Q^+ \coloneqq \text{cl}\,\text{conv} f(\mathcal X) + \mathbb R^p_{\geq 0}$, the \emph{Edgeworth-Pareto hull}.
We observe that $\mathcal Q^+$ is a closed convex polyhedron.
In our problem setting, we are interested in the facets of $\mathcal Q^+$.
We make the following assumptions about the problem:
\begin{enumerate}
    \item The value $\inf \left\{w^\mathsf T f(x) : x \in \mathcal X\right\}$ is computable for every $w\in \mathbb Q^p_{\geq 0}$ and has finite value. \label{aspts:computable}
    \item $\mathcal Q^+$ is finitely generated. \label{aspts:finite}
\end{enumerate}

From its definition and (\ref{aspts:computable}) it follows that $\mathcal Q^+$ is a rational polyhedron, $\text{rec}\, \mathcal Q^+ = \mathbb R^p_{\geq 0}$, and there is $y\in \mathbb R^p$ with $\mathcal Q^+ \subseteq y + \mathbb R^p_{\geq 0}$, i.e., an ideal point exists.

The assumptions do not generalize, i.e., there are MOPs that do not conform to these.
Examples include multi-objective linear programming instances without an ideal point (conflict with (\ref{aspts:computable})), but also continuous non-linear programming instances where $\mathcal Q^+$ exhibits curvature (conflict with (\ref{aspts:finite})).

On the other hand, the assumptions include \MOMILP iff an ideal point is present, and also non-linear problems, e.g., multi-objective combinatorial problems with quadratic constraints and objective functions.
Furthermore, non-convex problems can be solved in which curvature is present in the non-dominated set but not in $\mathcal Q^+$.

\subsection{Contribution}

In this work, we propose the first \OA algorithm which computes the extreme points and facets of $\mathcal Q^+$ for the problem defined in Section \ref{sec:problemdef}. Our algorithm is motivated by the \OA approach of \citet{benson1998outer} for \MOLP. It uses an \emph{oracle} which consists of solving single-objective \emph{weighted-sum problems}. We show that our algorithm needs a number of oracle calls which is polynomial in the number of facets of $\mathcal Q^+$ for a large class of problems.

As a consequence for \MOMILP, we are able to propose a similar theorem to the one by \cite{bokler2015output}: If the weighted-sum problem of a given \MOMILP is solvable in polynomial time, the facets can be computed with incremental-polynomial delay.
These results extend and complement the results of \cite{bokler2015output}, who showed that with an \IA\ algorithm, the extreme points can be found with incremental-polynomial delay in this case.

From a practical perspective, next to providing a lower bound set at any point of its execution, each oracle call may also produce a new Pareto optimal solution. This makes its use within multi-objective B\&B attractive also from the primal side. Finally, the algorithm also provides a valid lower bound set when the oracle is called with any relaxation of the original problem, such as the LP-relaxation (potentially augmented with valid inequalities) in case of \MOMILP. This opens perspectives for multi-objective branch-and-cut algorithms.

We observe that the problem we study also corresponds to the problem addressed in parametric integer linear programming with parametrization in the objective function \citep{geoffrion1977exceptional}, which is an interesting problem on its own. Moreover, a related problem is to compute the extreme points and/or facets of the convex hull of a linear projection of a mixed-integer problem onto a small subspace. Our algorithm is also capable of computing this projection. However, our focus in this paper will be on the optimization side.

\subsection{Outline}
In Section \ref{sec:notation} we provide notation, definitions and other preliminaries needed in the remainder of the paper. Section~\ref{sec:state-of-the-art} provides an overview of existing inner and outer approximation schemes. Section \ref{sec:oa} provides a general outline of \OA\  algorithms for multi-objective optimization. To work for a concrete problem class such as \MOLP\ or \MOMILP, a point separation oracle specific to the problem class is needed. In Section \ref{sec:oracles} we propose two generic separation oracles which only depend on the availability of a black-box algorithm for the single-objective problem of the considered problem class. The theoretical runtime of the algorithm when using the presented oracles is discussed in Section \ref{sec:run-time}. Section \ref{sec:comp} contains a computational study on instances from literature and also new non-linear instances, and also discusses implementation details. We give empirical results on the numerical accuracy of our approach and provide a comparison with \PolySCIP, which is an \IA solver for \MOMILP. Finally, Section \ref{sec:concl} concludes the paper.

\section{Notation, definitions and preliminaries \label{sec:notation}}

\subsection{General definitions} 

For $x,y\in \mathbb R^n$, the relations $x=y$, $x\leq y$, and $x<y$ apply component-wise.
We define $\mathbb R^n_{\geq 0} \coloneqq \{x \in \mathbb R^n : x \geq 0\}$, the non-negative orthant.
Let $A\in\mathbb Q^{m\times n}$, $C\in\mathbb Q^{p\times n}$, $b\in\mathbb Q^m$ and $n=n_1+n_2$. 
If $\mathcal X = \{x \in \mathbb Z^{n_1} \times \mathbb R^{n_2} : A x \geq b \}$ and $f(x) = Cx$, then the MOP is called \MOMILP.
If $n_1 = 0$, the problem is called a \MOLP and for $n_2 = 0$, it is called a \emph{multi-objective integer linear programming} problem (\MOILP).

We call the set $\mathcal Q \coloneqq f(\mathcal X)$ the \emph{feasible set in decision space}.
A point $y \in \mathcal Q$ is \emph{non-dominated}, iff there is no $\hat y\in\mathcal Q$ with $\hat y \leq y$ and $y\neq  \hat y$.
A feasible solution $x\in \mathcal X$ is called \emph{Pareto-optimal}, iff $f(x)$ is non-dominated.
A feasible solution $x$ is a \emph{supported Pareto-optimal solution}, iff there is a $w\in\mathbb R^p$ with $w>0$, such that $x$ is an optimal solution to the \emph{weighted-sum problem} or \emph{weighted-sum scalarization}:
 \begin{equation}
 (\text{WSUM}) \quad \min\left\{{w}^\mathsf T f(x): x \in \mathcal X\right\}
\label{eq:wsum} \notag 
\end{equation}
For every Pareto-optimal $x$, we call $f(x)$ a \emph{supported non-dominated} point.
Note that compared to \MOLP, in \MOMILP\ there can be \emph {non-supported} non-dominated points. The existence of such points justifies the need for multi-objective B\&B algorithms, with which they can be found.
This also justifies the fact that in MOLP the Edgeworth-Pareto hull usually has a different name:
It is called the \emph{upper image} of the MOLP instance.

\subsection{Polyhedra and faces} 

Let $P = \{Ax \geq b\}$ for $A \in \mathbb R^{m\times n}, b \in \mathbb R^m, n, m\in \mathbb N$ be a polyhedron.
Given $a\in \mathbb R^n$ and $\alpha \in \mathbb R$, the set $H^a_\alpha \coloneqq \{a^\mathsf T x = \alpha\}$ is called a \emph{hyperplane}.
The hyperplane $H^a_\alpha$ is \emph{supporting $P$ in $F$}, iff $F\coloneqq P \cap H^a_\alpha \neq \emptyset$ and for all $x\in P$, we have $a^\mathsf T x \geq \alpha$.
For every not necessarily supporting hyperplane $H^a_\alpha$, the set $F \coloneqq P \cap H^a_\alpha$ is called a \emph{face} of $P$.
A face with affine dimension $0$ is called an \emph{extreme point}, one with affine dimension $\dim P - 1$ is called a \emph{facet}.
Since for fully dimensional polyhedra, the hyperplane supporting a polyhedron in a facet is uniquely determined up to scaling, we often identify a facet with its defining (in)equality.

A vector $r\in \mathbb R^n$ is called a \emph{ray} of $P$, iff for every $x\in P$ and $\gamma > 0$ the point $x + \gamma r\in P$.
The set of all rays of $P$ is called its \emph{recession cone} or $\text{rec}\,P$.
If $\text{rec}\,P=\emptyset$, then we call $P$ a \emph{polytope}.
For the Edgeworth-Pareto hulls we observe in our model, we always have $\text{rec}\,\mathcal Q^+= \mathbb R^p_{\geq 0}$.

\subsection{Supported and extreme points} 
Supported non-dominated points can also be characterized by the Edgeworth-Pareto hull:
A point $y\in \mathcal Q$ is a supported non-dominated point iff it lies on the boundary of $\mathcal Q^+$.
We call the extreme points of $\mathcal Q^+$ the \emph{non-dominated extreme points} of the problem as well as the facets of $\mathcal Q^+$ the \emph{non-dominated facets} of the problem.
In our problem setting, we are interested in the facets of $\mathcal Q^+$ as they provide us with a lower bound set; 
as well as in Pareto-optimal solutions $x$ such that $f(x)$ is a non-dominated extreme point, as they provide us with primal solutions.
The polyhedron is illustrated in Figure \ref{fig:polytope} and the non-dominated extreme points are the red points with the black circles around them. 

\begin{figure}[tb]
\begin{center}
\begin{tikzpicture}[
scale=4.2,
axis/.style={very thick, ->, >=stealth'},
every node/.style={color=black}
]

\draw[fill=blue!50,draw=blue!50] (0.3,0.7) -- (0.2,0.9) -- (0.2,1.1)--(1.1,1.1)--(1.1,0.2)--(0.7,0.2)-- (0.5,0.4) --cycle;

\draw[draw=orange,fill=orange!50,thick] (0.2,0.9)  -- (0.7,0.6) -- (0.8,0.5) -- (0.7,0.2)  -- (0.5,0.4) -- (0.3,0.7) -- cycle;

\draw[step=.1cm,gray!50,very thin] (-0.05,-0.05) grid (1.05,1.05);
\draw[axis] (-0.1,0)  -- (1.1,0) node(xline)[right]
{$y_1$};
\draw[axis] (0,-0.1) -- (0,1.1) node(yline)[above] {$y_2$};

\draw[blue,thick](0.2,1.1)--(0.2,0.9)--(0.3,0.7)--(0.5,0.4) -- (0.7,0.2)  --(1.1,0.2) ;



\draw[fill=red,thick] (0.3,0.7) circle (0.5pt);
\draw[fill=red,thick] (0.2,0.9) circle (0.5pt);

\fill[red] (0.4,0.6) circle (0.5pt);
\fill[red] (0.4,0.7) circle (0.5pt);

\draw[fill=red,thick] (0.5,0.4) circle (0.5pt);
\fill[red] (0.5,0.5) circle (0.5pt);
\fill[red] (0.5,0.6) circle (0.5pt);

\fill[red] (0.6,0.4) circle (0.5pt);
\fill[red] (0.6,0.6) circle (0.5pt);

\fill[red] (0.7,0.4) circle (0.5pt);
\fill[red] (0.7,0.6) circle (0.5pt);

\draw[fill=red,thick] (0.7,0.2) circle (0.5pt);
\fill[red] (0.7,0.3) circle (0.5pt);
\fill[red] (0.8,0.5) circle (0.5pt);

\end{tikzpicture}
\end{center}
\caption{Exemplary illustration of the polytope $\textcolor{blue}{\mathcal Q^+}=\textcolor{orange}{\text{cl}\,\text{conv}}(\textcolor{red}{ {\mathcal Q}})+\mathbb R^2_\geq$ \label{fig:polytope}}
\end{figure}
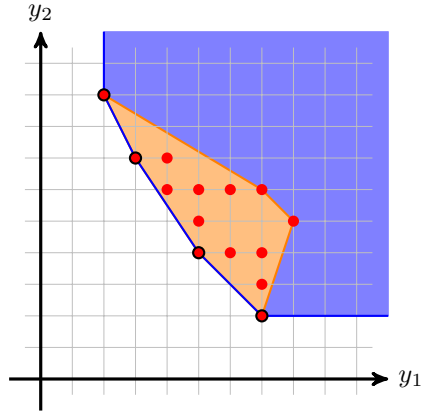


For more background on multi-objective optimization in general, we refer to, e.g., \cite{chinchuluun2007survey,ehrgott2005multicriteria}.

\subsection{Running times and output-sensitive complexity}

Since the number of facets and extreme points of the Edgeworth-Pareto hull (also called \emph{parametric complexity}) of many problems can grow exponentially, the traditional aim of a polynomial running time for algorithms computing this entities is not useful.
Instead, we usually regard an algorithm as efficient, if its running time can be bounded by a polynomial in the input and the output size.
We call such an algorithm an \emph{output-polynomial time} algorithm.

A more restrictive notion is an \emph{incremental-polynomial delay} algorithm:
The $k$th-delay of an enumeration algorithm is defined as the time between producing the $k$th and $(k+1)$th output, including the time to produce the first output and the time after the last output until termination.
In the incremental-polynomial delay model, we require the $k$th delay to be bounded by a polynomial in the input size and $k$.
In contrast to output-polynomial time algorithms, incremental-polynomial delay algorithms are not allowed to print all outputs at the end, but at certain time points there have to be updates.
However, the model still allows the delays between the outputs to grow depending on the size of the output so far.

This is remedied by the most restrictive model we discuss:
A \emph{polynomial delay} algorithm is an enumeration algorithm where every delay is bounded by a polynomial in the input size.

Naturally, a polynomial delay algorithm is an incremental-polynomial delay algorithm and an incremental-polynomial delay algorithm is also an output-polynomial algorithm but the reverse implications do not hold in general. See also \cite{BEMM16} for a more in-depth exposition of the subject.



\section{State-of-the-art}
\label{sec:state-of-the-art}

In this section, we provide an overview on existing \IA and \OA algorithms for multi-objective optimization problems.
A focus is on the state-of-the-art concerning finding the Edgeworth-Pareto hull of an \MOMILP as the best achievable linear lower bound set for an \MOMILP is the boundary of $\eghull$. At the end of the section, we also discuss the existing research on bound computation for multi-objective (mixed) integer non-linear programming.

Since $\eghull$ is a polyhedron, there are two general ways to represent it: By its generators, i.e., extreme points and rays, or by inequalities.
In the first representation, it suffices to compute the extreme points, since as we assume the feasible set to be bounded, the extreme rays of $\eghull$ are always the unit vectors.
In a multi-objective or parametric optimization setting, the extreme point representation is more compelling, since it provides us with value vectors of efficient solutions.
In our setting however, we are more interested in computing inequality representations.
In bi-objective problems, there is not much of a difference between these representations, but with a higher number of objectives, the computational cost of switching between them grows exponentially.


To the best of our knowledge, the first works that were concerned with computing extreme points of $\eghull$ of MOMILP were independently conducted by \cite{aneja1979bicriteria,cohon2004multiobjective}, and \cite{dial1979}.
Albeit, this algorithm, that is usually referenced as the \emph{dichotomic approach} today, does only work in the bi-objective case.
As the approach is well-known, we forgo describing it in detail.
Although the approach has been formulated for transportation problems, it is indeed a general purpose tool.
Much later, \cite{BEMM16} prove that the dichotomic approach can be implemented with polynomial delay.
Moreover, in the special case of MOLP, this can even be achieved with polynomial space in the input size using a slightly different algorithm also reported therein.

In the seminal work of \citet{benson1998outer}, the first outer approximation algorithm for solving \MOLP problems in objective space is given. Despite its name, it is an exact method. It forms the basis of the open source vector linear programming solver \bensolve \citep{lohne2017vector} and motivates our work. The key idea of Benson's algorithm is that $\mathcal Q^+$ can be approximated from the "outside", by adding new facets in each step, until the approximation $S$ is equal to $\mathcal Q^+$. Initially, the approximation $S = y + \mathbb R^p_{\geq 0}$, where $y$ is the ideal point, and an interior point $\hat{r} \in \text{int} \mathcal Q^+$ are required. Then, in each step a point that lies on the intersection of the boundary of $\mathcal Q^+$ and the line connecting one of the current vertices $y^*$ of $S$ and $\hat{r}$ is determined. It is used to obtain a supporting hyperplane of $\mathcal Q^+$, by solving an LP, or to prove that $y^*$ is a vertex of $\mathcal Q^+$.
\cite{bokler18} proves that this algorithm is capable of computing the extreme points of an MOLP in output-polynomial time and the facets even with incremental-polynomial delay.

Extending the work by Benson, \cite{ehrgott2012dual} develop a dual algorithm based on the geometric duality theory of MOLP by \cite{heyde2008geometric}.
Instead of operating in the decision or objective space, it operates in a dual space of the objective space.
Thus the dual algorithm essentially is a version of Benson's outer approximation algorithm with a very crucial caveat:
\cite{bokler2015output} show that instead of solving an involved LP to find separating hyperplanes, in this representation, it suffices to solve weighted sum LPs.
In spite of its dual description, the dual algorithm can be described as operating in the objective space in the following way:
The initial dual outer approximation is essentially a non-dominated extreme point of the problem with an attached non-negative orthant in the objective space, $\mathbb R^p_{\geq 0}$.
Due to duality theory, choosing an extreme point of the current dual outer approximation is the same as choosing a facet of the current inner approximation in the objective space.
Finding a new supporting hyperplane in dual space is again equivalent to finding a new supported point in the objective space.
In this view, the dual of Benson's algorithm is also the first instance of what we now call an inner approximation for MOLP.
Regarding theoretical running time, \cite{bokler18} shows that the dual algorithm computes the extreme points of an MOLP with incremental-polynomial delay and the facets in output-polynomial time.
Interestingly, the outer approximation is thus more efficient in computing facets and the inner approximation is more efficient in computing extreme points.


Very recently, \citet{csirmaz2020inner} introduced a unifying perspective on inner and outer approximation schemes for \MOLP, giving skeletal algorithms for both. Each scheme starts from an initial approximation $S$ of $\mathcal Q^+$. Then, the inner (outer) approximation scheme relies on a plane (vertex) separating oracle determining if a vertex (facet) is already part of $\mathcal Q^+$. If yes, it is marked at final, otherwise, $S$ is updated. Output to either scheme is $\mathcal Q^+$ in double description format (i.e., vertices and facets of $\mathcal Q^+$).
\citet{csirmaz2020inner} shows that Benson's algorithm and its variants follow the outer approximation scheme whereas the dual variant of \citet{ehrgott2012dual} falls into the described inner approximation scheme.

While Benson's algorithm and its variants address the multi-objective linear case, also some algorithms for obtaining the extreme supported points of 
\MOMILP have been proposed in the past years. To the best of our knowledge, none of them approximates $\mathcal Q^+$ from the outside and as such classifies as an \OA algorithm.

\citet{przybylski2010recursive} propose a recursive scheme that resorts to solving bi-objective problems via dichotomic search derived from the weights associated with extreme points of a given facet of the weight-set polytope. The weight-set is the set of all eligible weights in a weighted-sum scalarization. Decomposing it into weight-set components, for which a given extreme supported image is optimal, results in a weight-set decomposition. Computing this decomposition motivates  the development of the proposed scheme. \citet{przybylski2010recursive} show that their method works for three objectives and in theory also for more. An enumerative procedure 
is due to \citet{ozpeynirci2010exact}. It relies on what the authors call stages which are defined by $p$ points. In each step, the normal weight vector to the hyperplane $H$ defined by a given stage $R$ is used in the weighted sum scalarization of the MOMILP to obtain a new point $y$. If $y$ corresponds to one of the points of $R$, $R$ is identified as a facet defining stage. Otherwise, the new point $y$ is added to the set of non-dominated points and the list of to be explored stages is updated with $y$. A main drawback of the method is that it requires the generation of dummy points to initialize the search.

\citet{bokler2015output} propose \IA algorithms for MOCO problems. Their work builds on the dual variant of Benson's algorithm, observing that the LPs that need to be solved in the \MOLP case can be replaced by solving weighted sum problems of the original MOCO problem, showing that the geometric duality theory for MOLP by \cite{heyde2008geometric} can also be applied to non-convex problems if the interest lies in the Edgeworth-Pareto hull of the problem.
To overcome the drawback of enumerating redundant hyperplanes, they introduce a lexicographic variant which guarantees that only extreme points of $\mathcal Q^+$ are identified. First complexity results are also due to their work: \citet{bokler2015output} show that for every fixed number of objectives and MOCO with a polynomially solvable weighted-sum scalarization, the non-dominated extreme points can be computed with incremental-polynomial delay.
This result can be straightforwardly extended to \MOMILP under the assumption that an ideal point exists:
By the same arguments of the classic proof that (the decision version of) $\text{\MOMILP}\in\mathbf{NP}$, we see that the encoding lengths of optimal extreme point solutions of the integer hull remain polynomially bounded in the input size.
Optimizing over the integer hull of an \MOMILP thus is analogous to the combinatorial case.

\citet{borndorfer2016polyscip} develop a general purpose solver named \PolySCIP for MOILP and \MOLP with an arbitrary number of objectives as part of the constraint integer programming suite SCIP. To compute the extreme points of the MOILP and \MOLP, they rely on a lifted weight space polyhedron.
This polyhedron is essentially the dual polyhedron from geometric duality of MOLP applied to MOMILP, extending the work of \cite{bokler2015output}.
\PolySCIP employs a dual Benson outer approximation algorithm on this lifted weight space/dual space.
Hence as discussed above, \PolySCIP also falls in the category of inner approximations.
%
%
%
%

Even more recently, \citet{halffmann2020inner} have proposed an inner approximation algorithm for three-objective mixed integer linear programs, with the purpose of determining the weight-set decomposition. The proposed algorithm relies on solving bi-objective subproblems via the dichotomic approach to identify these components. It runs in output polynomial time.
Finally, building on the work of \citet{ozpeynirci2010exact} and \citet{przybylski2010recursive}, \citet{przybylski2019simple} have proposed a new \IA algorithm for \MOMILP. In each step, it generates the convex hull of the supported extreme points which have been identified so far. Then, a yet unexplored facet is chosen and used to define the next weighted sum scalarization to be solved. The solution may either be a new non-dominated supported point or identify the facet to be (partially)  part of $\mathcal Q^+$. The algorithm ends once all facets have been explored.




Multi-objective non-linear programming problems have received much less attention than linear problems. 
\citet{klamroth2003unbiased} develop inner and outer approximation algorithms producing piecewise linear approximations of the Pareto frontier for convex as well as nonconvex problems. 
An outer approximation algorithm for convex multi-objective programming problems with convex Pareto frontiers is due to \citet{ehrgott2011approximation}. It is an extension of Benson’s \OA algorithm and provides a set of weakly $\epsilon$-non-dominated points. Building on this work, \citet{lohne2014primal} extend Benson's outer approximation algorithm and its dual variant to convex vector optimization. 
\citet{niebling2019branch} use Benson cuts to improve the lower bound, which is initially based on the ideal point, in the context of a branch-and-bound scheme for smooth convex multi-objective optimization.  More recently, an outer approximation based B\&B algorithm has been proposed by \citet{de2020solving} for multi-objective mixed integer convex optimization problems. In the lower bound computation, building on the work of \citet{ehrgott2011approximation} and \citet{lohne2014primal}, a convex relaxation-based outer approximation is used. It relies on the computation of supporting hyperplanes. This computation is steered by the available local upper bounds within the B\&B algorithm. 



\section{The outer approximation algorithm \label{sec:oa}}
In this section, we first give a general outline of \OA\ algorithms for general multi-objective optimization problems.
To this end, we generalize the terminology introduced by \cite{csirmaz2020inner}.
The key ingredient in implementing such algorithms for concrete problem classes like \MOLP\ or \MOMILP\ is the specification and implementation of a \emph{point separating oracle} for the respective problem class. 
\begin{definition}[Point separating oracle \citep{csirmaz2020inner}]
A point separating oracle for a polytope $\mathcal Q^+ \subset \mathbb R^p$ is a black box algorithm which takes as input a point $y^* \in \mathbb R^p$ and returns as tuple $(status,H)$.
The output is as follows: i) $(inside, \emptyset)$, if $y^* \in \mathcal Q^+$, or ii) $(outside, H)$, where $H$ is a supporting hyperplane $H=\{y\in \mathbb R^p : w^T y = \alpha\}$ of $\mathcal Q^+$ such that $w^T y^* < \alpha$ and $w^T \widehat y \geq \alpha$ for each $\widehat y \in \mathcal  Q^+$ (i.e., $H$ separates $y^*$ from $\mathcal Q^+$).
\end{definition}

The \OA\ algorithm is described in Algorithm \ref{alg:outer}. 
The approximations $\mathcal S=\mathcal S_0 \supset \mathcal S_1 \supset  \ldots$ used within the algorithm are considered to be stored in a double description format, i.e., as vertices and facets.
The algorithm starts with an initial approximation $\mathcal S$.
This initial approximation consists of the \emph{ideal point} $y^I\in\mathbb Q^p$, whose $i$th component is the minimum value in the $i$th objective,
together with the 
non-negative orthant, i.e., $\mathcal S \coloneqq y^I + \mathbb R^p_\geq$. It proceeds in an iterative fashion by checking the vertices of the current approximation $\mathcal S_i$ for containment in $\mathcal Q^+$ using the separation oracle. This is done repeatedly until for some $\mathcal S_{i'}$ the oracle answers with status $inside$ for each vertex, which means $\mathcal S_{i'}=\mathcal Q^+$. 

\begin{algorithm}[h!tb]
\SetAlgoLined
\KwData{initial approximation $\mathcal S$ specified by double description (vertices and facets)}
\KwResult{$\mathcal Q^+$ specified by double description}
$\mathcal S_0=\mathcal S$\;
$insideVertices \gets \emptyset$\;
$i\gets 0$\;
\While{$\exists\ y^* \in vertices(\mathcal S_i): y^* \not \in insideVertices$}
{
$(status,H)\gets separationOracle(y^*,\mathcal S_i,\mathcal  Q^+)$\;
\eIf{$status=inside$}
{
$insideVertices \gets insideVertices \cup \{y^*\}$\;
}
{
$\mathcal S_{i+1}\gets \mathcal S_{i} \cap H$\;
$i \gets i+1$\;
}
}
$\mathcal Q^+ \gets S_i$\;
\caption{Generic \OA\ algorithm for multi-objective optimization \label{alg:outer}}
\end{algorithm}

Two steps of the \OA algorithm are illustrated in Figure~\ref{fig:outer}. It shows the image of all feasible solutions to an MOILP as green dots. Figure~\ref{subfig:Si} depicts the initial approximation $\mathcal S_0$, consisting of the ideal point $y^I$ and $\mathbb R_{\geq}^p$, in green. Then, the point separating oracle is called and the hyperplane $H$ is returned which separates $y^I$ from $\mathcal Q^+$, as shown in Figure~\ref{subfig:Si+1}. The new approximation $\mathcal S_{1}$ has two vertices depicted in blue $vertices(S_{1}) =  \{y^A, y^B\}$. An oracle call for the vertex filled with green, $y^A$, will return $inside$, since it belongs to $\mathcal Q^+$, whereas for the second vertex $y^B$, the oracle will produce a new hyperplane, separating $y^B$ from $\mathcal Q^+$.

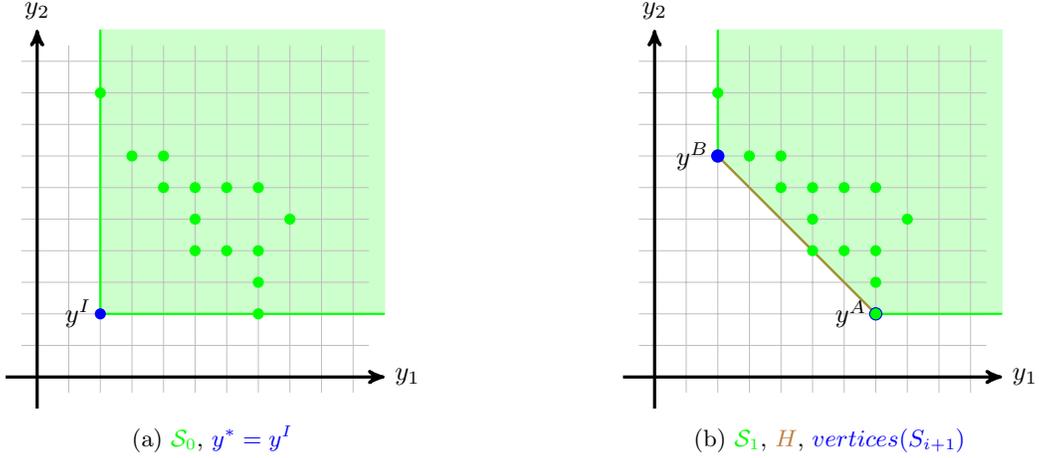
\begin{figure}[h!tb]
	\begin{center}
		\begin{subfigure}[h]{0.49\textwidth}
			\begin{center}
				\begin{tikzpicture}[
				scale=4.2,
				axis/.style={very thick, ->, >=stealth'},
				important line/.style={thick},
				dashed line/.style={dashed, thin},
				pile/.style={thick, ->, >=stealth', shorten <=2pt, shorten
					>=2pt},
				every node/.style={color=black}
				]
				
				\draw[fill=green!20,draw=green!20] (0.2,0.9) -- (0.2,1.1)--(1.1,1.1)--(1.1,0.2)--(0.7,0.2)--(0.2,0.2) --cycle;
				
				
				\draw[step=.1cm,gray!50,very thin] (-0.05,-0.05) grid (1.05,1.05);
				\draw[axis] (-0.1,0)  -- (1.1,0) node(xline)[right]
				{$y_1$};
				\draw[axis] (0,-0.1) -- (0,1.1) node(yline)[above] {$y_2$};
				
				\draw[green,thick](0.2,1.1)--(0.2,0.2)--(1.1,0.2) ;

				
				
				\fill[blue] (0.2,0.2) circle (0.5pt);
				\node[left] at (0.2,0.2) {$y^I$};
				
				\fill[green] (0.3,0.7) circle (0.5pt);
				\fill[green] (0.2,0.9) circle (0.5pt);
				
				\fill[green] (0.4,0.6) circle (0.5pt);
				\fill[green] (0.4,0.7) circle (0.5pt);
				
				\fill[green] (0.5,0.4) circle (0.5pt);
				\fill[green] (0.5,0.5) circle (0.5pt);
				\fill[green] (0.5,0.6) circle (0.5pt);
				
				\fill[green] (0.6,0.4) circle (0.5pt);
				\fill[green] (0.6,0.6) circle (0.5pt);
				
				\fill[green] (0.7,0.4) circle (0.5pt);
				\fill[green] (0.7,0.6) circle (0.5pt);
				
				\fill[green] (0.7,0.2) circle (0.5pt);
				\fill[green] (0.7,0.3) circle (0.5pt);
				
				\fill[green] (0.8,0.5) circle (0.5pt);
				
				\end{tikzpicture}
				\caption{\textcolor{green}{$\mathcal S_0$}, \textcolor{blue}{$y^* = y^I$} \label{subfig:Si}}
			\end{center}
		\end{subfigure}
		\begin{subfigure}[h]{0.49\textwidth}
			\begin{center}
				\begin{tikzpicture}[
				scale=4.2,
				axis/.style={very thick, ->, >=stealth'},
				important line/.style={thick},
				dashed line/.style={dashed, thin},
				pile/.style={thick, ->, >=stealth', shorten <=2pt, shorten
					>=2pt},
				every node/.style={color=black}
				]
				
				\draw[fill=green!20,draw=green!20](0.2,0.7) -- (0.2,0.9)-- (0.2,1.1)--(1.1,1.1)--(1.1,0.2)--(0.7,0.2)  --(0.5,0.4) --cycle;
				
				
				\draw[step=.1cm,gray!50,very thin] (-0.05,-0.05) grid (1.05,1.05);
				\draw[axis] (-0.1,0)  -- (1.1,0) node(xline)[right]
				{$y_1$};
				\draw[axis] (0,-0.1) -- (0,1.1) node(yline)[above] {$y_2$};
				
				\draw[green,thick](0.2,1.1)--(0.2,0.9) --(0.2,0.7) --(0.5,0.4)-- (0.7,0.2)  --(1.1,0.2) ;
				
				\draw[brown,thick] (0.2,0.7) -- (0.7,0.2) ;
				
				
				
				\fill[green] (0.2,0.7) circle (0.1pt);
				
				\fill[green] (0.3,0.7) circle (0.5pt);
				\fill[green] (0.2,0.9) circle (0.5pt);
				
				\fill[green] (0.4,0.6) circle (0.5pt);
				\fill[green] (0.4,0.7) circle (0.5pt);
				
				\fill[green] (0.5,0.4) circle (0.5pt);
				\fill[green] (0.5,0.5) circle (0.5pt);
				\fill[green] (0.5,0.6) circle (0.5pt);
				
				\fill[green] (0.6,0.4) circle (0.5pt);
				\fill[green] (0.6,0.6) circle (0.5pt);
				
				\fill[green] (0.7,0.4) circle (0.5pt);
				\fill[green] (0.7,0.6) circle (0.5pt);
				
				\fill[fill=blue,thick] (0.7,0.2) circle (0.6pt);
				\node[left] at (0.7,0.2) {$y^A$};
				
				\fill[green] (0.7,0.2) circle (0.5pt);
				\fill[blue] (0.2,0.7) circle (0.6pt);
				\node[left] at (0.2,0.7) {$y^B$};
				
				\fill[green] (0.7,0.3) circle (0.5pt);
				
				\fill[green] (0.8,0.5) circle (0.5pt);
				
				\end{tikzpicture}
				\caption{\textcolor{green}{$\mathcal S_{1}$}, \textcolor{brown}{$H$}, \textcolor{blue}{$vertices(S_{i+1})$}
				\label{subfig:Si+1}}
			\end{center}
		\end{subfigure}
		\caption{Two steps of the outer approximation algorithm
		\label{fig:outer}}
	\end{center}
\end{figure}

The convergence-behaviour of an \OA algorithm depends on the separation oracle.
In the original work by \cite{benson1998outer} for MOLP, the supporting hyperplanes to $\mathcal Q^+$ were only face supporting.
The first facet supporting hyperplane producing oracles were given in \cite{bokler18,csirmaz2020inner} for MOLP.
Hence, we are interested in facet supporting hyperplanes and we show in the following sections that our proposed general separation oracle produces facets of $\mathcal Q^+$. 

\begin{remark}
It is easy to see, that at any point of the algorithm, we have $\mathcal Q^+ \subseteq S_i$. Moreover, the algorithm still gives an \OA\ of $\mathcal Q^+$ when the separation oracle does not separate with respect to $\mathcal Q^+$ but to any superset $S(\mathcal Q^+) \supset \mathcal Q^+$. This means that $\mathcal Q^+$ can be replaced with any relaxation of the polyhedron $\mathcal Q^+$ in the separation oracle when the goal is to obtain lower bound sets for multi-objective B\&B algorithms.
\end{remark}

\begin{remark}
Let us briefly discuss the key running-time insights of the general OA algorithm.
Let us therefore assume, that a point separating oracle is available, produces facets only, and its running time is a constant as a simplification.
In this case it is easy to see that for every extreme point found, i.e., the situation of the first if-branch, we ask the point separating oracle once.
And for every facet found, i.e., the situation of the else-branch, we also ask the oracle once.
No facet and no extreme point is found twice.
Thus, we see that the overall running time is polynomial in the number of facets, extreme points, and the running time of the point separating oracle.
Moreover, every time we find a facet it is clear that we can output it directly.
This is not true for extreme points though, as the extreme points found in the if-branch are only intermediate extreme points and may be deleted in later iterations.
The real extreme points of the Edgeworth-Pareto hull are only know at the time the algorithm terminates.
This is the main reason why the algorithm is considered more efficient in computing facets than extreme points.
Consequently, a finite number of extreme points and facets is needed unless we want to approximate the Edgeworth-Pareto hull.
For an in-depth look into the running-times of the inner and outer approximation algorithms see \cite{bokler18}.
\end{remark}

\section{New separation oracles \label{sec:oracles}} 

\subsection{Geometric duality}

The techniques we use in this paper are heavily influenced by the theory on MOLP.
For MOLP, a duality theory \citep{heyde2008geometric} exists that is built around a polyhedron $\mathcal D$, the lower image, that is geometrically dual to $\mathcal Q^+$.
For $w\in \mathbb R^{p-1}$, we define $\lambda(w) \coloneqq \left(w_1, \dots, w_{p-1}, 1 - \sum_{i=1}^{p-1} w_i\right)\in \mathbb R^p$.
For an MOLP with $A\in\mathbb Q^{m\times n}, b\in\mathbb Q^m$, and $C\in\mathbb Q^{p\times n}$, the \emph{lower image} is defined as
\begin{multline*}
\mathcal D \coloneqq \bigg\{ \left(w_1, \dots, w_{p-1}, \alpha \right) \in \mathbb R ^ p :\\
w_i\geq 0, \sum_{i=1}^{p-1} w_i\leq 1, u^\mathsf T A = \lambda(w)^\mathsf T C, u \in \mathbb R ^m, \alpha \leq b^\mathsf T u\bigg\}.
\end{multline*}
We can characterize $\mathcal D$ independently of $A, b$, and $C$ and only with respect to the upper image polyhedron $\mathcal Q^+$ in the following way:
\[\mathcal D = \left\{ \left(w_1, \dots, w_{p-1}, \alpha \right) \in \mathbb R^p : w_i \geq 0, \sum_{i=1}^{p-1} w_i \leq 1, y \in \mathcal Q^+:\alpha \leq \lambda(w)^\mathsf T y \right\}.\]
In this fashion, we give a more general definition of the lower image for any MOP with our additional requirements:
\begin{definition}
    Let $\mathcal Q^+$ be the Edgeworth-Pareto hull of an instance of a MOP with the above requirements.
    The \emph{generalized lower image} is defined as:
    \[\mathcal D \coloneqq \left\{ (w_1, \dots, w_{p-1}, \alpha ) \in \mathbb R^p : w_i \geq 0, \sum_{i=1}^{p-1} w_i \leq 1, y \in \mathcal Q^+:\alpha \leq \lambda(w)^\mathsf T y \right\}\]
\end{definition}

Observe that every such finitely generated polyhedron $\mathcal Q^+$ can be perceived as the upper image of some MOLP, e.g., the decision polyhedron is also $\mathcal Q^+$ and the objective matrix is the identity.
Thus, we can apply the geometric duality theory to general Edgeworth-Pareto hulls as defined above.
In particular, there exists an inclusion reversing one-to-one map from the $k$-dimensional faces of $\mathcal Q^+$ to the $(p-k-1)$-dimensional faces of $\mathcal D$.
For our purpose, the most important corollary is the following:
For every extreme point $(\widehat{w}_1,\dots, \widehat{w}_{p-1}, \widehat{\alpha})$ of $\mathcal D$, the hyperplane $\left\{y\in \mathbb R ^ p : \lambda(\widehat w)^\mathsf T y = \widehat\alpha\right\}$ is facet supporting to $\mathcal Q^+$.

\subsection{A first general separation oracle}
%
Let $w_i$ be the coefficient of objective $i, i=1, \ldots, p$ in the desired hyperplane $H$ and $\alpha$ be the right-hand-side of $H$.
\begin{align}
(\text{Sep-}y^*)\quad \min\ (y^*)^\mathsf T w -\alpha & \quad &  \notag \\
s.t.\ y^\mathsf T w -\alpha &\geq 0 & \forall y \in \mathcal Q \label{eq:feas} \tag{FEAS} \\
\sum_{i=1,\ldots, p} w_i &= 1 \label{eq:wsum1b} \tag{WSUM1}   \\
w_i &\geq 0 & i=1,\ldots, p \label{eq:wsum2} \tag{WSUM2} 
\end{align}

Problem $(\text{Sep-}y^*)$ encodes that $(w,\alpha)$ should induce a separating hyperplane by enforcing that all $y \in \mathcal Q$ should be on the positive side of the hyperplane using constraints \eqref{eq:feas}. Constraint \eqref{eq:wsum1b} is a normalisation constraint for the coefficients of the obtained hyperplane. It will be used in the proofs later on. If the objective function value of $(\text{Sep-}y^*)$ is negative for a given $y^*$, we get that $y^* \not \in \mathcal Q^+$ and $(w,\alpha)$ gives the coefficients of the corresponding separating hyperplane. To deal with the infinite size of $\mathcal Q$, we propose to solve $(\text{Sep-}y^*)$ using a cutting-plane approach, where constraints \eqref{eq:feas} get separated: Given a solution $(\widehat w,\widehat \alpha)$, there exists a violated constraint \eqref{eq:feas}, iff $\inf\left\{{\widehat w}^\mathsf T y : y \in \mathcal Q\right\} = \inf\left\{{\widehat w}^\mathsf T f(x) : x \in \mathcal X\right\} < \widehat \alpha$. Thus, the separation problem for \eqref{eq:feas} consists in solving a weighted-sum problem.

\begin{remark}
In our implementation, the weighted-sum problems are weighted-sum scalarizations of the given original mixed-integer problems. We thus obtain integer feasible solutions. Further implementation details are given in Section~ \ref{sec:impl}. 
\end{remark}

We note that the separation oracle ($\text{Sep-}y^*$) shares some similarity with the LPs used in local cut separation for single-objective mixed-integer linear programming \citep{applegate2001tsp,chvatal2013local}.



\ifx
\begin{align}
(\text{D-Sep-}y^*)\quad  \max {\mu} &\quad  & \notag \\
s.t. \sum_{i: \hat y_i \in \mathcal Q} \lambda_i \hat y_i  +\mu e  &\leq y^* \label{eq:dfeas} \tag{D-FEAS}  \\
 \sum_{i: \hat y_i \in \mathcal Q} \lambda_i &= 1 \label{eq:sum1} \tag{SUM1} \\
 \lambda_i &\geq 0 & i: \hat y_i \in \mathcal Q\label{eq:sum2}\tag{SUM2} 
\end{align}

\begin{remark}
Problem $(\text{Sep-}y^*)$ can be modified by replacing the normalisation constraint \eqref{eq:wsum1b} with constraints $w_i \leq 1$, $i=1,\ldots, p$. Let $\mu_i$ be the corresponding dual multiplier for these new constraints. The objective function in the dual then reads $\sum_{1\leq i \leq p} \mu_i$ and the constraints have a term $\sum_{1\leq i \leq p} \mu_i e_i$, where $e_i$ is the $i$-th unit vector, instead of $\mu e$.
\end{remark}
\fi

We now prove that this is an implementation of a general point separating oracle and that new separating hyperplanes are facet supporting.
\begin{lemma}\label{lemma:point-sep}
    Let $y^*\in\mathbb R^p$ and $(\widehat w,\widehat \alpha)$ be an optimal solution to Sep-$y^*$.
    We have $(y^*)^\mathsf T \widehat w - \widehat \alpha < 0$ iff $y^*\notin \mathcal Q^+$.
\end{lemma}
\begin{proof}
  If $(y^*)^\mathsf T \widehat w - \widehat \alpha < 0$, then $\widehat w^\mathsf T (y^*) < \widehat \alpha$ and $\widehat w^\mathsf T y \geq \widehat \alpha$ for all $y\in \mathcal Q^+$ by constraints (\ref{eq:feas}).
  In other words, $H^{\widehat w, \widehat \alpha} \coloneqq \left\{y\in \mathbb R^p : \widehat w^T y = \widehat \alpha\right\}$ separates $y^*$ from $\mathcal Q^+$ and thus $y^*\notin \mathcal Q$.
  If $y^*$ is not in $\mathcal Q^+$, then by the hyperplane separation theorem, there is a hyperplane $H^{\widehat w, \widehat \alpha}$, separating $y^*$ from $\mathcal Q^+$, i.e., $\widehat w^\mathsf T (y^*) < \alpha$ and $\widehat w^\mathsf T y > \alpha$ for all $y\in \mathcal Q^+$.
  Moreover, we can assume $\widehat w$ to be normalized, i.e., $\sum_{i=1}^p \widehat w_i = 1$.
  As $\text{rec}\, \mathcal Q^+ = \mathbb R^p_{\geq 0}$, we can also assume all $w_i$ to be at least $0$.
  Consequently, $(\widehat w, \widehat \alpha)$ is a feasible solution to Sep-$y^*$ and its value $(y^*)^\mathsf T \widehat w - \widehat \alpha$ is at most $0$ and an upper bound on the optimal value.\qedhere
\end{proof}
\begin{lemma}
    For every optimal extreme point solution $(\widehat w, \widehat \alpha)$ of Sep-$y^*$, we have \[H^{\widehat w, \widehat \alpha} \coloneqq \left\{y\in\mathbb R^p : \widehat w^\mathsf T y = \widehat \alpha\right\}\] is facet supporting for $\mathcal Q^+$.
\end{lemma}
\begin{proof}
By setting $\bar \alpha \coloneqq (y^*)^\mathsf T w - \alpha$ and using that $\text{rec}\, \mathcal Q^+ = \mathbb R^p_{\geq 0}$, we can rewrite the feasible set of Sep-$y^*$ as follows:
\[\left\{(w_1, \dots, w_{p}, \bar \alpha) : \forall y \in \mathcal Q^+: w^\mathsf T (y- y^*) + \bar \alpha\geq 0, w_i \geq 0, \sum_{i = 1}^{p} w_i = 1\right\}\]
As $w_p = 1 - \sum_{i=1}^{p-1}w_i$, finding the smallest $\bar \alpha$ in the above set is equivalent to finding the largest $\bar\alpha$ in
\[\left\{(w_1, \dots, w_{p-1}, -\bar \alpha) : \forall y \in \mathcal Q^+ - y^*: \lambda(w)^\mathsf T y \geq \bar \alpha, w_i \geq 0, \sum_{i = 1}^{p-1} w_i \leq 1\right\}\]
which is the generalized lower image of $\mathcal Q^+ - y^*$.
An optimal extreme point solution of Sep-$y^*$ is thus an extreme point of the corresponding generalized lower image.
The claim thus follows from the geometric duality theorem.\qedhere
\end{proof}

\begin{theorem}
    The above algorithm using (Sep-$y^*$) is a point separating oracle that provides us with facet supporting inequalities.
\end{theorem}

\subsection{A target cut-like separation oracle}
Another separation oracle can be defined as follows. This oracle follows the \emph{target-cut} paradigm introduced in \cite{buchheim2008local} for single-objective mixed-integer linear
programming. In the description below, we assume that all points encountered in the separation process are $\geq 1$ (component-wise). This is without loss of generality, as we can always add a large constant to all points. In Section \ref{sec:impl} we describe how we do this in our implementation. In the presentation of this oracle we start from the dual side and then proceed to the primal side. 
\begin{align}
(\text{D-TSep-}y^*)\quad \max \sum_{i: \hat y_i \in \mathcal Q} \lambda_i  & \quad & \notag \\
s.t. \sum_{i: \hat y_i \in \mathcal Q} \lambda_i \hat y_i &\leq y^* \label{eq:dtfeas} \tag{D-T-FEAS}  \\
 \lambda_i &\geq 0 & i: \hat y_i \in \mathcal Q\label{eq:tsum2}\tag{TSUM} 
\end{align}
Constraints \eqref{eq:dtfeas} impose that $y^*$ must be equal to a point, or dominated by a point which can be obtained by a linear combination of the points in $\mathcal Q$. Constraints \eqref{eq:tsum2} ensure that this linear combination only has positive coefficients. The key insight is that $y^*$ only belongs to $\mathcal Q^+$ iff the objective of $\text{(D-TSep-}y^*)$ is one or larger, see \cite{buchheim2008local} for more details.
\begin{align}
(\text{TSep-}y^*)\quad \min\ (y^*)^\mathsf T w& \quad &  \notag \\
s.t.\ y^\mathsf T w &\geq 1 & \forall y \in \mathcal Q \label{eq:tfeas} \tag{TFEAS} \\
w_i &\geq 0 & i=1,\ldots, p \label{eq:rwsum} \tag{TWSUM} 
\end{align}
Using the above ideas, we can implement a point separating oracle as follows.
If an optimal solution $\widehat w$ of (TSep-$y^*$) has an objective value smaller than $1$, ${\widehat w}^\mathsf T y \geq 1$ gives a separating hyperplane.
Otherwise, we return \emph{inside}. Similar to constraints \eqref{eq:feas} of $(\text{Sep-}y^*)$, we propose to solve $(\text{TSep-}y^*)$ using a cutting-plane approach, where constraints \eqref{eq:tfeas} are separated by solving weighted-sum problems.

The reason this algorithm is a point separating oracle is analogous to Lemma \ref{lemma:point-sep}:
We see that for any feasible solution $\widehat w$ with objective value $(y^*)^\mathsf T \widehat w < 1$, the hyperplane $\{x \in \mathbb R^p : w^\mathsf T x = 1\}$ again separates $y^*$ from $\mathcal Q$, as long as $\mathcal Q \subseteq \mathbb R^p_\geq$. Moreover, in case the minimum value is at least $1$, then there is no separating hyperplane, hence $y^*\in \mathcal Q^+$. 

To prove that an optimal extreme point solution corresponds in fact to a facet supporting inequality, we observe that the feasible set of TSep-$y^*$ is a polar dual polyhedron to $\mathcal Q^+$.
\begin{theorem}
    The above algorithm using (TSep-$y^*$) is a point separating oracle that provides us with facet supporting inequalities.
\end{theorem}

\subsection{Theoretical running time\label{sec:run-time}}

Let us first analyze the most general version of the OA algorithm.
Using the methodology of \cite{bokler18}, we observe the following corollary:
\begin{corollary}
    The \OA algorithm with the point separating oracles from Section \ref{sec:oa} produces the next facet in time $O\Big(k^{\lfloor\frac{p}{2}\rfloor} (T_O + k\log k)\Big)$, where $T_O$ is the running time of the oracle and $k\in\mathbb N$ is the number of facets already computed.
\end{corollary}

Both our new point separating oracles show that a point separating oracle can be implemented by solving a sequence of weighted-sum scalarizations.
The major advantage of solving weighted-sum scalarizations over other known methods to implement the oracle model is that no new constraints are added and thus the structural properties of the feasible sets remain untouched.
It remains to show, however, that the number of weighted-sum scalarizations that need to be solved do not grow too fast.
The problems Sep-$y^*$ and TSep-$y^*$ are linear programming problems and thus only weakly polynomial-time algorithms are known to solve them.
Hence, we need to make some assumptions about the encoding length of their input, while we first avoid to restrict the running time of the weighted-sum oracle.
\begin{lemma}\label{lemma:poly-weighted}
    An optimal extreme point of (Sep-$y^*$) and (TSep-$y^*$) can be computed with a number of weighted-sum calls that grows polynomially in the encoding length of the extreme points of $\mathcal Q^+$ if an upper bound $\varphi \in \mathbb N$ on the maximum of these encoding lengths is available.
\end{lemma}
\begin{proof}
    Using the methodology of \cite{grotschel1993}, we compute an optimal extreme point solution to the LPs Sep-$y^*$ or TSep-$y^*$ in polynomial time by using the ellipsoid method and employing a rounding mechanism to reach an optimal extreme point solution.
    The ellipsoid method has a running time polynomial in the number of variables ($p+1$) and the encoding length of extreme points of the input polyhedron, i.e., the feasible set of Sep-$y^*$ or TSep-$y^*$.
    The encoding length of the extreme points of the feasible sets of Sep-$y^*$ and TSep-$y^*$ are polynomially bounded in the encoding length of their facets.
    The facets of Sep-$y^*$ and TSep-$y^*$, in turn, correspond to the extreme points of $\mathcal Q^+$.
    The bound on the encoding length $\varphi$ is therefore necessary for the ellipsoid method to bound the volume of the feasible set Sep-$y^*$ and TSep-$y^*$ and thus the number of its iterations.
    
    The separation oracle in the invokation of the ellipsoid method is implemented by a solver for the weighted-sum problem.
    The weighted-sum solver is hence called only polynomially many times.\qedhere
\end{proof}

For a large problem class the extreme points of $\mathcal Q^+$ can be bounded by a polynomial in the input size, namely MOMILPs:
It is well-known that the encoding length of every extreme point of the integer hull is bounded by a polynomial in the input size.
The encoding length of the image of each such feasible extreme point under the objective function is thus also polynomially bounded.
In fact, the objective functions do not need to be linear; polynomial computability in extreme points of the integer hull suffices.
In all these cases, $\varphi$ in the above lemma can be estimated using these bounds.

\begin{corollary}\label{lemma:momilp-poly-weighted}
    For MOMILP, an optimal extreme point of (Sep-$y^*$) and (TSep-$y^*$) can be computed with a polynomial number of weighted-sum calls.
\end{corollary}

We now proceed and add the restriction of a polynomial-time solvable weighted-sum problem.
In this case, the whole point separation can be performed in polynomial time for any MOMILP and the following theorem follows from the above.
\begin{theorem}
    If the weighted-sum problem for a given MOMILP is polynomial time solvable, then the facets of $Q^+$ can be computed with incremental-polynomial delay.
\end{theorem}

This theorem complements the work by \cite{bokler2015output}, where the same result was shown for linear combinatorial optimization problems and extreme point of $\mathcal Q^+$.
It is easy to see, that the result can also be applied to MOMILPs.

\section{Computational experiments \label{sec:comp}}
%

The proposed \OA\ algorithms are implemented in Python using CPLEX 12.10 as LP/mixed integer (non-)linear programming-solver and the \texttt{parma polyhedral library (ppl)}\sloppy
\citep{bagnara2008parma} for vertex enumeration. In this section, we first describe the instances used in our computational experiments in Section \ref{sec:inst}, followed by a discussion of implementation details in Section \ref{sec:impl} and results in Section \ref{sec:performance}.


\newcommand{\MAPinst}{\texttt{AP-MOO}\xspace}
\newcommand{\MKPinst}{\texttt{KP-MOO}\xspace}
\newcommand{\QCPinst}{\texttt{QCP-MOO}\xspace}
\newcommand{\IPinst}{\texttt{IP-MOO}\xspace}
\newcommand{\MIPinst}{\texttt{MIP-MOO}\xspace}
\newcommand{\MIPperinst}{\texttt{MIP-PER}\xspace}
\newcommand{\moolibrary}{\texttt{moolibrary}\xspace}

\subsection{Instances\label{sec:inst}}

We use instance sets from the literature and also new non-linear instances. Our newly created instances are available online under \url{https://msinnl.github.io/pages/multi-objective.html}. The sets are as follows.

\begin{itemize}
\item \MAPinst: These are multi-objective assignment instances from the multi-objective optimization library (\moolibrary)\footnote{available at \url{http://home.ku.edu.tr/~moolibrary/}}, with three objectives. The number of agents and tasks is identical and ranges from 5 to 40, in increments of 5. The objective function coefficients are random integers in the range $[1,20]$. There are 100 instances in the set. They were proposed in \cite{kirlik2014}; similar instances were also used in experiments for \IA algorithms in \citep{ozpeynirci2010exact,przybylski2010recursive,przybylski2019simple}.
\item \MKPinst: These are multi-objective knapsack instances from the \moolibrary with three to five objectives. Both profits and weights are random integers from the interval $[1,1000]$. The budget is calculated as half the total weight of all items, rounded up to the next integer. There are 160 instances in the set: 100 instances with three objectives (10 to 100 items), 40 with four (10 to 40 items), and 20 with five (10 to 20 items). Like the set \MAPinst, these instances are proposed in \cite{kirlik2014} and similar instances are also used in experiments for \IA algorithms in \citep{ozpeynirci2010exact,przybylski2010recursive,przybylski2019simple}.

\item \QCPinst: These are three-objective quadratic covering instances newly created for this work. The problem is defined as follows:

\begin{align*}
\min &(x^\mathsf T R^1 x,x^\mathsf T R^3 x,x^\mathsf T R^3 x) \\
 w^\mathsf T x &\geq C \\
x &\in \{0,1\}^n
\end{align*}

where $R^j=V_j^{\mathsf T} V_j \in \mathbb Z^{n \times n}$ for $j \in \{1,2,3\}$, $w \in \mathbb Z^n$ and $C\in \mathbb Z$. We have created ten instances for each $n \in \{10,20,30,40\}$. The entries of $V_j$ are random integers from the interval $[1,10]$ and the entries of $w$ are random integer from the interval $[1,100]$. The value of $C$ is set to $\lfloor w^\mathsf T x/4 \rfloor$. The
structure of the problem allows an easy linearization of the non-linear objective
function using a standard McCormick-linearization to transform the problem
into an \MOMILP. This allows us to analyze potential benefits of solving the problem directly as a non-linear multi-objective problem against solving the linearized version of it.



\end{itemize}



\subsection{Implementation details\label{sec:impl}}

\paragraph{Instance transformation and offset calculation.}
We first transform the instances to minimization form by flipping the objective coefficients for all objectives which are maximization. We then compute $y^I$ by solving $\min_{y \in \mathcal Q} e_i^T y$ for unit-vector $e_i$. If the value of $y^I$ is negative for an objective $j$, we add the constant $-y^I_j+1$ to the respective objective to ensure that all points encountered in the separation process are component-wise $\geq 1$ as needed for (TSep-$y^*$). Moreover, instead of using the value of one on the right-hand-side of constraints \eqref{eq:rwsum}, we fix the value to the sum of the components of the ideal point (after adding the constant $-y^I_j+1$ to objectives $j$ with negative value). This proved to be more numerically stable in preliminary computations.

\paragraph{Separation.} For the initial approximation $\mathcal S$, the ideal point $y^I$ can be obtained by solving $\min_{y \in \mathcal Q} e_i^T y$ for unit-vector $e_i$. Moreover, we initialize the separation LPs by adding constraints \eqref{eq:feas}, resp., \eqref{eq:tfeas} induced by the solutions obtained for the $p$ problems used for calculating the ideal point. When separating constraints \eqref{eq:feas}, resp., \eqref{eq:tfeas}, we use tolerance $\varepsilon$ for checking violation. Once a constraint \eqref{eq:feas}, resp., \eqref{eq:tfeas} is added to the separation LP, we leave it there for the remainder of the algorithm (i.e., for all subsequent separation oracle calls). In the separation LPs, the \texttt{numerical emphasis} parameter of CPLEX is turned on, and the feasiblity and optimality tolerances are set to $1\text{e-}9$, i.e., the most accurate value possible. When checking the result of the separation oracle (i.e., the objective value of the separation LP against zero, resp., one), we also use tolerance $\varepsilon$. 

We use \texttt{ppl} with integer numbers as input, as using fractional numbers leads to (more) numerical instabilities. In order to do so, we scale each $(w,\alpha)$ obtained from the separation oracle by $10^9$ and take the integer part of the obtained number. When checking if a point $y^* \in vertices(\mathcal S_i)$ is in $\mathcal Q^+$, we proceed as follows: When a point $y^*$ is within tolerance value $\varepsilon$ to a point in $y' \in insideVertices$ for each coordinate, we consider $y^*=y'$ and do not call the separation oracle for it. For each $\mathcal S_i$, we check the points in the order in which they are made available by \texttt{ppl}. Moreover, when updating the outer approximations, we use a slightly different strategy compared to the outline in Algorithm \ref{alg:outer}: We do not immediately calculate $\mathcal S_i \cap H$ and move to $\mathcal S_{i+1}$ once we discovered a $y^* \not \in insideVertices$; instead, we check all $y^* \in vertices(\mathcal S_i)$, collect the obtained separating hyperplanes, and then add them all at the same time to obtain the next outer approximation. This approach turned out to be faster and more numerically stable in preliminary computations. Finally, we also add every point found when separating \eqref{eq:feas}, resp., \eqref{eq:tfeas} to $insideVertices$, as these points are obtained from the solution of a weighted-sum problem, and thus are supported non-dominated points.

\subsection{Results\label{sec:performance}}

\paragraph{Numerical accuracy of our approach and the influence of $\varepsilon$.}

The focus in the implementation of our algorithms is to find a good balance between numerical accuracy and speed. We note that \IA algorithms can suffer from numerical issues, for example in \cite{przybylski2019simple} different versions of the same \IA algorithm are tested,
and some versions cannot find the complete set of non-dominated extremes points for some instances. In our computational study, we use the open-source \IA solver \PolySCIP \citep{borndorfer2016polyscip,maher2017scip} for comparison. As experiments show (see below), there are some instances were \PolySCIP does not find all extreme points; moreover, we observed in \PolySCIP's output that in some cases it also reports some weakly-dominated points as part of the set of non-dominated extremes points. In \OA, numerical issues can lead to missing cut-off points, which do not belong to $\mathcal Q^+$. However, even if we miss some of these points, the obtained solution is still a valid outer approximation. 


To investigate the influence of different values of $\varepsilon$ on the performance and accuracy of our approach, we first consider the instance set \MAPinst. Since the assignment problem has a totally unimodular constraint matrix, it can be solved as a linear programming problem and thus we can obtain the extreme points of $\mathcal Q^+$ by using the multi-objective LP solver \bensolve. Table \ref{ta:ap} shows the number of extreme points of $\mathcal Q^+$ obtained by our \OA with both oracles (columns $|\tilde{\mathcal Q}^+|$), and the number of extreme points when these points are rounded to the next integer (columns $|\mathcal Q^+|$, recall that \MAPinst has integer coefficients, thus all extreme points have integer values). We considered $\varepsilon=1e-3$ and $\varepsilon=1e-5$ for both oracles. We also report the number of extreme points obtained by \bensolve and \PolySCIP and the run time of the methods (column $t[s]$). The results are aggregated by instance size. To allow for a meaningful comparison, we only report results for instance sizes were all approaches terminate within the time limit for all instances of this size.

The results show that for instances with size up to 30, all methods give a consistent number of extreme points. For instances of size 35, the size of $|\mathcal Q^+|$ obtained by the \OA is consistent with \bensolve considering both oracles when $\varepsilon=1e-5$. For instances of size 40, the number of extreme points obtained by the \OA and by \bensolve becomes inconsistent, but only by a small amount. This is not surprising as with larger instances size, the size of $|\mathcal Q^+|$ also grows, and thus the potential for numerical inaccuracies increases as well. We observe that the number of non-rounded extreme points ($|\tilde{\mathcal Q}^+|$) also stays fairly consistent for smaller-sized instances, and for all sizes is about three times as large as the number of points after rounding. Regarding the run time, using $\varepsilon=1e-3$ is about twice as fast as using $\varepsilon=1e-5$.

\begin{landscape}
\begin{table}[ht]
\centering
\caption{Results for instances \MAPinst aggregated by size \label{ta:ap}} 
\begin{tabular}{l|rrr|rrr|rrr|rrr|rr|rr}
  \toprule
  size & \multicolumn{6}{|c|}{Sep-$y^*$} &\multicolumn{6}{c}{TSep-$y^*$}  &\multicolumn{2}{c}{\texttt{bensolve}}  &\multicolumn{2}{c}{\texttt{PolySCIP}}    \\  & \multicolumn{3}{c|}{$\varepsilon=1e-3$} & \multicolumn{3}{c|}{$\varepsilon=1e-5$} & \multicolumn{3}{c|}{$\varepsilon=1e-3$} & \multicolumn{3}{c|}{$\varepsilon=1e-5$}  & & \\ 
  & $|\tilde{\mathcal Q}^+|$ & $|\mathcal Q^+|$ & t [s] & $|\tilde{\mathcal Q}^+|$ & $|\mathcal Q^+|$ & t [s] & $|\tilde{\mathcal Q}^+|$ & $|\mathcal Q^+|$ & t [s] & $|\tilde{\mathcal Q}^+|$ & $|\mathcal Q^+|$ & t [s] & $|\mathcal Q^+|$& t [s] & $|\mathcal Q^+|$& t [s] \\  \midrule
5 & 19.0 & 7.5 & 0.7 & 19.0 & 7.5 & 0.6 & 18.7 & 7.5 & 0.1 & 18.7 & 7.5 & 0.1 & 7.5 & 0.0 & 7.5 & 1.0 \\ 
  10 & 117.4 & 39.0 & 2.0 & 117.4 & 39.0 & 1.8 & 117.4 & 39.0 & 0.8 & 117.4 & 39.0 & 0.9 & 39.0 & 0.5 & 38.1 & 2.3 \\  
  15 & 258.7 & 83.1 & 3.5 & 258.7 & 83.1 & 3.9 & 258.7 & 83.1 & 3.0 & 258.7 & 83.1 & 3.6 & 83.1 & 0.4 & 78.5 & 6.3 \\  
  20 & 495.3 & 161.3 & 10.8 & 495.3 & 161.3 & 12.6 & 495.3 & 161.3 & 10.2 & 495.3 & 161.3 & 13.6 & 161.3 & 0.4 & 153.5 & 11.1 \\ 
  25 & 809.1 & 253.1 & 26.9 & 809.1 & 253.1 & 36.6 & 808.7 & 253.1 & 27.2 & 809.1 & 253.1 & 39.2 & 253.1 & 0.8 & 223.3 & 17.6 \\ 
  30 & 1183.0 & 379.4 & 62.6 & 1183.2 & 379.4 & 95.7 & 1182.4 & 379.5 & 63.1 & 1183.2 & 379.4 & 101.6 & 379.4 & 1.7 & 348.1 & 35.5 \\ 
  35 & 1579.7 & 501.2 & 121.7 & 1581.1 & 501.4 & 202.0 & 1576.9 & 500.9 & 123.1 & 1581.1 & 501.4 & 213.3 & 501.4 & 3.2 & 458.0 & 65.9 \\ 
  40 & 2170.0 & 698.9 & 244.8 & 2160.4 & 700.8 & 437.2 & 2167.8 & 698.4 & 247.5 & 2164.6 & 700.5 & 451.9 & 699.1 & 5.8 & 631.9 & 122.2 \\  
   \bottomrule
\end{tabular}
\end{table}
\end{landscape}

\paragraph{Results for instance set \MKPinst.}

Table \ref{ta:kp} shows the results for instance set \MKPinst
obtained by our algorithms and by \PolySCIP. Our algorithms are run with $\varepsilon=1e-5$. In these tables we give the number of rounded extreme points (columns $|\mathcal Q^+|$), the number of facets (columns \#fac) and the runtime (columns t [s]). We see that the number of points is quite similar for all three methods, however, as the instances become larger for $p=3$ and for $p=4,5$ in general, the numbers differ. This is not unexpected, as both larger instances and more objectives could have a negative effect on numerical accuracy. In particular, as for larger objectives, the number of facets is growing considerably. We observe that for the number of facets, the difference is larger compared to the number of points. This could be explained by the fact that $|\mathcal Q^+|$ is the rounded set of points, which exploits that the instance coefficients are integer. Thus, numerical inaccuracies are corrected by the rounding, e.g., multiple slightly fractional points are likely to be rounded to the same integer point. The runtime-performance of both separation oracels is quite similar.

\begin{table}[h!tb]
\centering
\caption{Results for instances \MKPinst with $p=3,4,5$ aggregated by number of items \label{ta:kp}} 
\begin{tabular}{lrrr|rrr|rr}
  \toprule
  size & \multicolumn{3}{c|}{Sep-$y^*$} &\multicolumn{3}{c|}{TSep-$y^*$} &\multicolumn{2}{c}{\texttt{PolySCIP}}   \\   & \#fac & $|\mathcal Q^+|$  & t [s] & \#fac  & $| \mathcal Q^+|$
  & t [s]  & $|\mathcal Q^+|$  & t [s]  \\  \midrule
  10 & 11.1 & 5.0 & 0.2 & 11.1 & 5.0 & 0.1 & 5.0 & 7.8 \\ 
  20 & 30.2 & 14.1 & 0.4 & 29.9 & 14.1 & 0.3 & 14.1 & 4.4 \\ 
  30 & 56.4 & 26.2 & 0.9 & 56.6 & 26.2 & 0.8 & 26.1 & 3.5 \\ 
  40 & 77.6 & 36.3 & 1.5 & 78.0 & 36.3 & 1.4 & 35.9 & 3.6 \\ 
  50 & 105.8 & 48.4 & 2.2 & 106.1 & 48.6 & 2.0 & 48.7 & 2.5 \\ 
  60 & 150.9 & 69.5 & 3.7 & 150.9 & 69.5 & 3.2 & 69.7 & 3.0 \\ 
  70 & 201.4 & 92.5 & 5.7 & 203.2 & 92.7 & 5.0 & 92.5 & 2.4 \\ 
  80 & 252.4 & 114.4 & 8.1 & 249.7 & 114.5 & 6.8 & 114.3 & 3.1 \\ 
  90 & 311.2 & 141.9 & 11.4 & 309.7 & 141.6 & 9.1 & 141.8 & 3.5 \\ 
  100 & 388.9 & 177.7 & 16.2 & 391.6 & 177.7 & 12.9 & 176.7 & 7.4 \\ 
  \midrule
10 & 27.2 & 6.6 & 0.2 & 26.5 & 6.6 & 0.2 & 6.6 & 0.7 \\ 
  20 & 169.7 & 30.2 & 2.2 & 170.0 & 30.2 & 2.0 & 30.0 & 1.1 \\ 
  30 & 378.7 & 61.2 & 7.2 & 375.6 & 61.1 & 5.9 & 60.2 & 1.4 \\ 
  40 & 908.3 & 136.7 & 25.5 & 915.1 & 136.9 & 19.2 & 133.5 & 2.5 \\ 
  \midrule
  10 & 108.2 & 10.2 & 0.9 & 109.3 & 10.1 & 0.9 & 10.0 & 0.5 \\ 
  20 & 661.4 &38.9 & 15.1 & 650.7 & 38.9 & 11.6 & 38.1 & 1.0 \\ 
   \bottomrule
\end{tabular}
\end{table}

\paragraph{Results for instance set \QCPinst.}

Table \ref{ta:qcp} shows the results for instance set \QCPinst
obtained by our algorithms in both the variant where the problem is solved directly as a multi-objective mixed-integer non-linear programming problem, and also the variant where the problem is linearized, i.e., where a \MOMILP gets solved. This means that the black-box solver for the weighted-sum problems must either solve mixed integer non-linear programming problems or mixed integer linear programming problems. We note that CPLEX is capable of both. We report the same values as in Table \ref{ta:kp} and additionally report the number of instances solved out of ten for each item-size (columns \#sol). We see that for these instances, the oracle Sep-$y^*$ seems more efficient, as both in the non-linear as well as in the linearized case two more instances can be solved. By looking at the item-sizes where all instances can be solved (i.e., 10 and 20) the numerical accuracy does not seem to be affected by the type of black-box solver used. On the other hand, the results show directly solving the non-linear programming variant of the problem instead of the linearized is generally faster. This means it could pay off that our proposed algorithm is capable of directly solving multi-objective mixed-integer non-linear programming problems without the need for linearization.

\begin{landscape}

\begin{table}[ht]
\centering
\caption{Results for instances \QCPinst aggregated by number of items \label{ta:qcp}} 
\begin{tabular}{lrrrr|rrrr|rrrr|rrrr}
  \toprule
  & \multicolumn{8}{c|}{non-linear} &\multicolumn{8}{c}{linearized} \\
  size & \multicolumn{4}{c|}{Sep-$y^*$} &\multicolumn{4}{c|}{TSep-$y^*$} & \multicolumn{4}{c|}{Sep-$y^*$} &\multicolumn{4}{c}{TSep-$y^*$}   \\   & \#fac & $|\mathcal{Q^+}|$ & t [s]  & \#sol & \#fac  & $|\mathcal{Q^+}|$ & t [s]  & \#sol & \#fac  & $|\mathcal{Q^+}|$ & t [s]  & \#sol& \#fac  & $|\mathcal{Q^+}|$& t [s]  & \#sol  \\  \midrule
10 & 8.0 & 3.9 & 0.2 & 10 & 8.0 & 3.9 & 0.2 & 10 & 8.0 & 3.9 & 0.2 & 10 & 8.0 & 3.9 & 0.2 & 10 \\ 
  20 & 14.3 & 5.7 & 0.7 & 10 & 14.3 & 5.7 & 0.7 & 10 & 14.3 & 5.7 & 8.8 & 10 & 14.3 & 5.7 & 8.6 & 10 \\ 
  30 & 20.7 & 8.0 & 122.0 & 8 & 21.0 & 8.1 & 181.5 & 7 & 21.0 & 8.1 & 154.3 & 8 & 21.0 & 8.1 & 207.1 & 7 \\ 
  40 & 42.3 & 14.0 & 480.3 & 2 & 42.3 & 14.0 & 540.2 & 1 & 41.2 & 14.2 & 488.2 & 2 & 41.9 & 14.1 & 543.6 & 1 \\ 
   \bottomrule
\end{tabular}
\end{table}
\end{landscape}



\section{Conclusion}
\label{sec:concl}
We present the first outer-approximation algorithm to compute the extreme points and facets for multi-objective mixed-integer linear programming problems and also certain multi-objective mixed-integer non-linear programming problems.
We show that the number of weighted-sum mixed-integer oracles needed to compute these extreme points and facets can be bounded polynomially by the size of the output and the input for \MOMILP.
And for \MOMILP with polynomial-time computable weighted-sum scalarizations, the facets of the Edgeworth-Pareto hull can be computed with incremental polynomial delay.
We provide a computational study on instances from literature and new non-linear instances. In this study, we investigate the numerical accuracy of our method and also include a comparison with \PolySCIP, which is an inner approximation algorithm for multi-objective integer linear programming problems. 
We conjecture that existing and future multi-objective B\&B methods can benefit from partial bound set computation which now is possible with our proposed method. Future work should investigate the efficacy of this approach in practical B\&B algorithms.
\newline
\newline
\textbf{Data availability} The datasets generated during and/or analyzed during the current study are available in the web pages \url{http://home.ku.edu.tr/~moolibrary/} and \url{https://msinnl.github.io/pages/multi-objective.html}.

\ifArXiV
\section*{Acknowledgments}{
This research was funded in whole, or in part, by the Austrian Science Fund (FWF) [P 31366, 35160-N]. The research was also supported by the Linz Institute of Technology (Project LIT-2019-7-YOU-211) and the JKU Business School.
For the purpose of open access, the author has applied a CC BY public copyright licence to any Author Accepted Manuscript version arising from this submission.
}
\fi

\bibliographystyle{spbasic} 
\bibliography{oa.bib} 

\end{document}